\documentclass{icmart}
\contact[vostrik@uoregon.edu]{Department of Mathematics,
University of Oregon, Eugene, OR 97403, USA}
\usepackage{amssymb}
\usepackage{amscd}
\usepackage{amsfonts}
\usepackage{latexsym}
\usepackage[all]{xy}
\usepackage{verbatim}

\newtheorem{theorem}{Theorem}[section]
\newtheorem{corollary}[theorem]{Corollary}
\newtheorem{lemma}[theorem]{Lemma}

\newtheorem{conjecture}[theorem]{Conjecture}

\theoremstyle{definition}
\newtheorem{definition}[theorem]{Definition}
\newtheorem{remark}[theorem]{Remark}
\newtheorem{example}[theorem]{Example}

\newcommand{\id}{\text{id}} 
 
\newcommand{\Ker}{\text{Ker\,}}

\newcommand{\Fun}{\text{Fun}}

\newcommand{\Coh}{\text{Coh}}
\newcommand{\Sk}{\text{Sk}}
\newcommand{\Wh}{\text{Wh}}
\newcommand{\Spr}{\text{Spr}}

\renewcommand{\Vec}{\text{Vec}}
\newcommand{\Ann}{\text{Ann}}
\newcommand{\Pic}{\text{Pic}}

\newcommand{\Hom}{\text{Hom}}

\newcommand{\Out}{\text{Out}}

\newcommand{\Rep}{\text{Rep}}

\newcommand{\ev}{\text{ev}}
\newcommand{\coev}{\text{coev}}

\newcommand{\cC}{\mathcal{C}}
\newcommand{\cD}{\mathcal{D}}

\newcommand{\cZ}{\mathcal{Z}}
\newcommand{\cH}{\mathcal{H}}

\newcommand{\cA}{\mathcal{A}}

\newcommand{\cN}{\mathcal{N}}

\newcommand{\cO}{\mathcal{O}}

\newcommand{\be}{\mathbf{1}}

\newcommand{\fg}{\mathfrak{g}}
\newcommand{\fk}{\mathfrak{k}}
\newcommand{\fp}{\mathfrak{p}}
\newcommand{\fq}{\mathfrak{q}}

\newcommand{\m}{\mathfrak{m}}

\renewcommand{\be}{\mathbf{1}}

\newcommand{\BZ}{{\mathbb Z}}
\newcommand{\BC}{{\mathbb C}}
\newcommand{\BO}{{\mathbb O}}
\newcommand{\BQ}{{\mathbb Q}}
\newcommand{\BR}{{\mathbb R}}

\newcommand{\Fl}{{\mathcal B}}

\newcommand{\cM}{{\mathcal M}}
\newcommand{\cJ}{{\mathcal J}}
\newcommand{\cP}{{\mathcal P}}
\newcommand{\cU}{{\mathcal U}}

\newcommand{\ot}{\otimes}

\title[Multi-fusion categories of Harish-Chandra bimodules]{Multi-fusion categories of Harish-Chandra bimodules}

\author[Victor Ostrik]{Victor Ostrik}

\begin{document}
\begin{abstract} We survey some results on tensor products of irreducible Harish-Chandra
bimodules. It turns out that such tensor products are semisimple in suitable Serre quotient
categories. We explain how to identify the resulting semisimple tensor categories and describe
some applications to representation theory.
\end{abstract}



\maketitle

\section{Introduction}
The notion of tensor category is ubiquitous in representation theory. A classical example is the
theory of Tannakian categories (see \cite{SR, De}) which shows that a linear algebraic group can be
recovered from the tensor category of its finite dimensional representations. The tensor product in
a Tannakian category is commutative in a very strong sense. In this note we will be interested
in tensor categories for which the tensor product is not assumed to be commutative. One reason
for the relevance of such categories to representation theory is very simple: the category of
bimodules over an arbitrary algebra $A$ is a tensor category with tensor product given by tensoring
over $A$ and this tensor product is non-commutative in general.

A classical notion in the representation theory of a complex semisimple Lie algebra $\fg$ 
is that of {\em Harish-Chandra bimodules}.
These objects were introduced by Harish-Chandra \cite{HC} in order to reduce some questions
of continuous representation theory of complex semisimple groups (considered as real Lie
groups) to pure algebra. A number of deep results on Harish-Chandra bimodules are known,
see e.g. \cite{BG, Ja, So}. In this note we will be interested in just one aspect of the theory, namely,
in the structure of the tensor category of Harish-Chandra bimodules. Some significant steps
towards a complete description of this category were made in \cite{So}, however due to non-semisimplicity
this description is necessarily quite complicated. It turns out that the classical notion of
{\em associated variety} (see e.g. \cite{Jo, Vo}) provides us with a kind of filtration on this category;
moreover one can define an ``associated graded'' category with respect to this filtration which is much
simpler than the original category but still carries an important information about the category
of Harish-Chandra bimodules. Thus one defines certain interesting semisimple subquotients 
of the tensor category Harish-Chandra bimodules associated with various nilpotent
orbits in $\fg$ which we call {\em cell categories}, see Section \ref{cellc}. The idea of this definition
can be traced back to the work of Joseph \cite{Jotc} and the name is justified
by the connection with the theory of Kazhdan-Lusztig cells \cite{KL}. A nice property of the cell
categories is that they are {\em multi-fusion} in a sense
of \cite{ENO}. The known results from the theory of multi-fusion categories turned out to be powerful
enough to identify these categories with some categories constructed from some finite groups. 
This is interesting in its own right but also gives a better understanding of some notions of
representation theory such as Lusztig's quotients and Lusztig's subgroups.  

One can hope to apply the ideas above in the following way.
The Harish-Chandra bimodules act on various categories of $\fg-$modules via tensoring over the
universal enveloping algebra of $\fg$. We can exploit this action restricted to the semisimple
subquotients as above in order to obtain interesting information about such categories of $\fg-$modules.
One example of such application is the theory of finite $W-$algebras  where this strategy
allowed to obtain the information on the number of finite dimensional simple modules,
see Section \ref{W}. It was suggested by Bezrukavnikov that similar approach might work 
for Harish-Chandra modules, see Section \ref{hcm}.

The cell categories described above can be realized via truncated convolution of some perverse sheaves 
on the flag variety associated with $\fg$. One advantage of this description is a greater flexibility. For
example we can replace the complex semisimple Lie algebra $\fg$ by a semisimple algebraic group
$G$ defined over a field of arbitrary characteristic. Moreover, using the tensor categorical construction
of the {\em Drinfeld center} one connects the cell categories with the theory of {\em character sheaves}
on $G$, see \cite{BN,BFO2,Lutrun}. We describe briefly these developments in Section \ref{CSh}.

This paper is organized as follows. In Section \ref{mfc} we review briefly some notions of the
theory of tensor categories. In Section \ref{HCb} we introduce the Harish-Chandra bimodules and
define the cell categories. In Section \ref{W} we explain how to use Whittaker modules and
Premet's $W-$algebras in order to establish some basic properties of the cell categories.
Conversely we show that the actions of the cell categories can be used in order to get an
information about finite dimensional representations of $W-$algebras. Finally in Section \ref{CSh}
we describe the interaction of the cell categories and some classes of sheaves on algebraic
varieties associated with $\fg$.

My understanding of the subject described in this note was deeply influenced by collaborations with
Roman Bezrukavnikov, Pavel Etingof, Michael Finkelberg, Ivan Losev, and Dmitri Nikshych. It is
my great pleasure to express my sincere gratitude and appreciation to them. Thanks are also due
to Jonathan Brundan, Victor Ginzburg, and Ivan Losev for useful comments on a preliminary version of this paper.

\section{Multi-fusion categories} \label{mfc}
\subsection{Monoidal categories} For the purposes of this note a {\em monoidal category} is a 
quadruple $(\cC, \ot, a, \be)$
where $\cC$ is a category, $\ot : \cC \times \cC \to \cC$ is a bifunctor (called {\em tensor product}), $a$ is a
natural isomorphism $a_{X,Y,Z}: (X\ot Y)\ot Z\simeq X\ot (Y\ot Z)$ (so $a$ is called {\em associativity
constraint}), $\be \in \cC$ is an object (called {\em unit object})
such that the following axioms hold:

1. Pentagon axiom: the following diagram commutes for all $W,X,Y,Z\in \cC$:
$$\xymatrix{
&((W\ot X)\ot Y)\ot Z  \ar[dl]^{a_{W,X,Y}\otimes \id_Z}
\ar[dr]_{a_{W\otimes X,Y,Z}}& \\
(W \ot (X \ot Y))\ot Z \ar[d]^{a_{W, X\ot Y, Z}} && 
(W \ot X) \ot (Y \ot Z) \ar[d]_{a_{W, X, Y\ot Z}}  \\
W \ot ((X\ot Y) \ot Z)  \ar[rr]^{\id_W \ot a_{X,Y,Z}} &&
W \ot (X \ot (Y \ot Z)) 
}$$

2. Unit axiom: both functors $\be \ot ?$ and $?\ot \be$ are isomorphic to the identity functor.

It is well known that this definition of monoidal category reduces to the traditional one 
(see e.g. \cite{Mac}) if we  fix an isomorphism $\be \ot \be \simeq \be$. 

Further one defines natural notions of tensor functors and tensor equivalences, see e.g. \cite{Mac}.
From a practical point of view tensor equivalent monoidal categories are indistinguishable. 

A basic example of monoidal category is category of $R-$bimodules over a ring (with unity) $R$. 
In this case
the tensor product is tensor product $\ot_R$ over $R$, the unit object is $R$ considered as a bimodule,
and the associativity constraint is the obvious one. A closely related example is a category of endofunctors of a category with tensor product given by the composition. Also modules
over a {\em commuative} ring form a monoidal category, most familiar example being the
category of vector spaces over a field $k$.

Here is another more abstract example. 
\begin{example}[\cite{Sin}] \label{Sinh}
Let $A$ be a group and let $S$ be an abelian group, both written multiplicatively.
We consider the category where the objects are elements of $A$ and the morphisms are given by
$\Hom(g,h)=\emptyset$ if $g\ne h$ and $\Hom(g,g)=S$ for any $g$. We have a bifunctor
$g\ot h=gh$ and $\alpha \ot \beta =\alpha \beta$ for objects $g,h \in A$ and morphisms 
$\alpha,\beta \in S$. The associativity constraint amounts to a morphism
$\omega_{g,h,k}\in \Hom(ghk, ghk)=S$ for any three elements $g,h,k\in A$. One verifies
that the pentagon axiom reduces to the equation $\partial \omega =1$ which says that $\omega$
is a 3-cocycle on $A$ with values in $S$. Moreover, any 2-cochain $\psi$ determines a tensor
structure on the identity functor between tensor categories with associativity constraints given
by 3-cocycles $\omega$ and $\omega \cdot \partial \psi$. We see that monoidal structures on
our category are parameterized by the cohomology group $H^3(A,S)$.
\end{example}

We explain now that nontrivial associativity constraints do appear in tensor categories
of bimodules. 
\begin{example} \label{out}
Let $R$ be an algebra over a field $k$ with trivial center. Recall that $R-$bimodule $M$
is {\em invertible} if there exists a $R-$bimodule $N$ such that $M\ot_RN\simeq N\ot_RM\simeq R$.
The invertible bimodules form a tensor category with respect to $\ot_R$ (morphisms being the {\em
isomorphisms} of bimodules). This category is equivalent to the category of type described in 
Example \ref{Sinh}.
The group of automorphisms of any object is $k^\times$, so 
the associativity constraint determines a class $\omega \in H^3(\Pic(R),k^\times)$ where $\Pic(R)$ is the group
of isomorphism classes of invertible $R-$bimodules (this is the non-commutative
{\em Picard group} of $R$).
This class is often nontrivial. Indeed for any automorphism $\phi$ of $R$ we can define 
invertible bimodule $R_\phi$ as follows: $R_\phi =R$ as a vector space and the action
is given by $(a,b)\cdot c:=ac\phi(b)$. The bimodule $R_\phi$ is isomorphic to $R$ if and only if
$\phi$ is inner, so we get a well known embedding $\Out(R)\subset \Pic(R)$. Now assume that
$\phi$ is outer automorphism such that $\phi^2$ is inner, that is $\phi^2(x)=gxg^{-1}$ for some
invertible element $g\in R$ and all $x\in R$; 
thus $\phi$ generates subgroup $\BZ/2\BZ \subset \Out(R)\subset \Pic(R)$.
 It is easy to see that then $\phi(g)=\pm g$; we leave it to the reader to check that the restriction
 of the class $\omega$ to $\BZ/2\BZ$ is nontrivial if and only if $\phi(g)=-g$. Here is an example
 when this is the case:
 $$R=\BC \langle g,x,y\rangle /(xy-yx-1, g^2-1, gx+xg, gy+yg), \;\; \phi(g)=-g, \phi(x)=-y, \phi(y)=x.$$
 \end{example} 
 
 \begin{remark} For a commutative algebra $R$ one defines the category of invertible $R-$modules
 similarly to Example \ref{out}. Since we have canonically $M\ot_RN\simeq N\ot_RM$, this category
 has an additional structure of {\em symmetric tensor category}. This implies that the cohomology
 class representing the associativity constraint (see Example \ref{Sinh}) is always trivial, see 
 \cite[Section 7]{Dugger}.
 \end{remark}

A crucially important technical assumption on a monoidal category $\cC$ is that of {\em rigidity}. We recall
that for an object $X\in \cC$ its {\em right dual} is an object $X^*\in \cC$ together with {\em evaluation}
and {\em coevaluation} morphisms $\ev_X: X^*\ot X\to \be$ and $\coev_X: \be \to X\ot X^*$ such that
the composition
$$X \xrightarrow{\coev_X\ot \id_X} (X\ot X^*)\ot X 
\xrightarrow{a_{X,X^*,X}}
X\ot (X^*\ot X) \xrightarrow{\id_X\ot \ev_X} X$$
equals the identity morphism and the composition
$$X^* \xrightarrow{\id_{X^*}\ot \coev_X} X^*\ot (X\ot X^*) \xrightarrow{a^{-1}_{X,X^*,X}} 
 (X^*\ot X)\ot X^* \xrightarrow{\ev_X \ot \id_{X^*}} X^* $$
equals the identity morphism. Similarly, a {\em left dual} of $X$ is ${}^*X\in \cC$ such that $X$ is right
dual of ${}^*X$. A monoidal category $\cC$ is {\em rigid} if any object of $\cC$ has both left and right
duals. 

\begin{example} A vector space considered as an object of the monoidal category of vector spaces
has left or right dual if and only if it is finite dimensional. A bimodule over an algebra $A$ has
a right dual if and only if it is finitely generated projective when considered as a left $A-$module.
The category considered in Example \ref{Sinh} is always rigid.
\end{example}

\subsection{Semisimplicity}
We will fix an algebraically closed field $k$ of characteristic zero.
We recall that a $k-$linear category $\cC$ is called {\em semisimple} if 
there is a collection $\{ L_i\}_{i\in J}$ of objects of $\cC$ such that

(i) $\dim_k\Hom(L_i,L_j)=\delta_{ij}=\left\{ \begin{array}{cc}1&\mbox{if}\; i=j;\\ 0&\mbox{if}\; i\ne j.\end{array}\right.$

(ii) any object of $\cC$ is isomorphic to a finite direct sum of the objects $L_i$ and any direct
sum (including the empty one) is contained in $\cC$.

For example if $\cA$ as a $k-$linear abelian category with finite dimensional spaces of morphisms
the category of semisimple objects in $\cA$ (that is the full subcategory consisting of direct sums of
simple objects) is semisimple.

Clearly the isomorphism classes of the objects $L_i$ are uniquely determined by the category $\cC$.
These objects are {\em simple objects} of $\cC$ (note that zero object is not simple).

For a semisimple category $\cC$ let $K(\cC)$ be its {\em Grothendieck group}; this is a free abelian
group with basis $[L_i]_{i\in J}$. For any $M\in \cC$ we have its {\em class}:
$$[M]:=\sum_{i\in J}\dim_k\Hom(L_i,M)[L_i]\in K(\cC).$$

We say that a semisimple category is {\em finite} if the isomorphism classes of simple objects
form a finite set.

\begin{definition} A {\em multi-fusion category} over $k$ is a $k-$linear rigid monoidal category 
which is finite semisimple. A {\em fusion category} is a multi-fusion category such that the unit
object is simple.
\end{definition}

Let $\be =\oplus_{i\in I}\be_i$ be the decomposition of the unit object of a multi-fusion category
into the  sum of simple objects.
One shows that the objects $\be_i$ are ``orthogonal idempotents'', that is $\be_i\otimes \be_j\simeq 
\delta_{ij}\be_i$. For any simple object $L\in \cC$ there are unique $i,j\in I$ such that
$\be_i\ot L=L=L\ot \be_j$. We say that a multi-fusion category $\cC$ is {\em indecomposable}
if for any pair $i,j\in I$ there exists a simple $L$ with $L=\be_i\ot L\ot \be_j$. One shows that
any multi-fusion category naturally decomposes into a unique direct sum of indecomposable ones.

The Grothendieck group of a multi-fusion category has a natural structure of a ring 
via $[M]\cdot [N]=[M\ot N]$. The {\em Grothendieck ring} $K(\cC)$ of a multi-fusion category
$\cC$ together with its basis consisting of the classes of simple objects is a {\em based ring}
in the sense of \cite{Lulead}.

\begin{remark} The Grothendieck ring $K(\cC)$ of a multi-fusion category $\cC$ together with
its basis determines the tensor product and the unit object in $\cC$ uniquely up to
isomorphism. Thus the only part of information describing $\cC$ and not contained in $K(\cC)$
is the associativity constraint.
\end{remark}

\begin{example} \label{multex}
(i) Let $A$ be a finite group. Consider a fusion category where simple objects $L_g$
are labeled by $g\in A$ and $L_g\ot L_h\simeq L_{gh}$. Similarly to Example \ref{Sinh} the possible
associativity constraints in this category are classified by $H^3(A, k^\times)$. We will denote
such a category with associativity constraint given by $\omega \in H^3(A, k^\times)$ by 
$\Vec_A^\omega$; we set $\Vec_A=\Vec_A^{\omega_0}$ where $\omega_0$ is the neutral element
of $H^3(A, k^\times)$.
The Grothendieck ring $K(\Vec_A^\omega)$ is the group ring $\BZ[A]$ with a basis
$\{ g\}_{g\in A}$.

(ii) Let $R$ be a finite dimensional semisimple $k-$algebra, e.g. $R=k\oplus k$. Then the category 
of finite dimensional $R-$bimodules with tensor product $\ot_R$ is a multi-fusion category.
Its Grothendieck ring is the ring of matrices over $\BZ$ with a basis consisting of matrix units.

(iii) Let $A$ be a finite group and $Y$ is a finite set on which $A$ acts. Consider
the category $\Coh_A(Y\times Y)$ of finite dimensional $A-$equivariant vector bundles 
on the set $Y\times Y$. This category has a natural {\em convolution tensor product} defined
as follows.
Let $p_{ij}: Y\times Y\times Y\to Y\times Y, i,j\in \{ 1,2,3\}$ be the various projections; then for
$F_1, F_2\in \Coh_A(Y\times Y)$ we set $F_1*F_2=p_{13*}(p_{12}^*(F_1)\otimes p_{23}^*(F_2))$
(here $p_{ij*}$ and $p_{ij}^*$ are the functors of direct and inverse image and $\ot$ is the pointwise
tensor product). Then the bifunctor $*$ has a natural associativity constraint and thus
$\Coh_A(Y\times Y)$ is a multi-fusion category.
In the special case of trivial $A$ we get example (ii) above; in the case when $Y$ consists of one
point we get the category $\Rep(A)$ of representation of $A$; if the action of $A$ on $Y$ is free
and transitive we get category $\Vec_A$ from (i).
\end{example}

\subsection{Module categories and dual categories} \label{mdc}
Let $\cC$ be a monoidal category and let $\cM$ be a category. We say that $\cC$ {\em acts} on $\cM$
if we are give a tensor functor from $\cC$ to the category of endofunctors of $\cM$. Equivalently,
we have a bifunctor $\cC \times \cM \to \cM, (X,M)\mapsto X\ot M$ endowed with the natural associativity 
isomorphism $(X\ot Y)\ot M\simeq X\ot (Y\ot M)$ such that the counterpart of the pentagon axiom holds
and the functor $\be \ot ?: \cM \to \cM$ is isomorphic to the identity functor. Thus in such situation 
we often say that $\cM$ is a {\em module category} over $\cC$. Further one defines module functors
between module categories and, in particular, the equivalences of module categories, see \cite{Omc}.

{\bf Convention.} In the case when both categories $\cC$ and $\cM$ are $k-$linear we will consider
only $k-$linear actions. If $\cC$ is a multi-fusion category all module categories over $\cC$ are assumed
to be finite semisimple and non-zero.

\begin{example} \label{modex}
(i) Let $Y$ be a finite set with an action of a finite group $A$. Then 
the category $\cM =\Coh(Y)$ of finite dimensional vector bundles on $Y$ has an obvious structure of
module category over $\Vec_A$. It is easy to recover the $A-$set $Y$ from $\cM$: the set $Y$
is just the set of isomorphism classes of simple objects in $\cM$ and the action of $A$ is recovered 
by considering the action of simple objects of $\Vec_A$ on simple objects of $\cM$.

(ii) We can generalize the example above to the category
$\Vec_A^\omega$ from Example \ref{multex} (i). Pick a cocycle $\tilde \omega$ representing $\omega$.
Let $B\subset A$ be a subgroup and
let $\psi$ be a 2-cochain on $B$ such that $\partial \psi =\tilde \omega |_B$. Then $\psi$ determines 
an multiplication morphism $R_B\ot R_B\to R_B$ where $R_B=\oplus_{g\in B}L_g$; moreover
this morphism makes $R_B$ into {\em associative algebra in the category} $\Vec_A^\omega$. 
Let $\cM=\cM(B,\psi)$ be the category of right $R_B-$modules in the category $\Vec_A^\omega$; 
then the left tensoring with object of $\Vec_A^\omega$ makes $\cM$ into module category
over $\Vec_A^\omega$. Note that the simple objects of $\cM$ are naturally labeled by the cosets
$A/B$; moreover the action of simple objects of $\Vec_A^\omega$ on simple objects of $\cM$ 
is the same as the action of $A$ on $A/B$. Thus we consider the module category $\cM(B,\psi)$ as 
a cohomologically twisted version of the action of $A$ on $A/B$. More generally one can consider
a direct sum of module categories of the form $\cM(B,\psi)$; this is a twisted version of the action
of $A$ on a finite set. We will use for such module categories the notation $\cM=\Coh(Y)$ where it
is understood that the ``set'' $Y$ carries the cohomological information describing the module
category $\cM$ (thus $Y$ is completely determined when a finite collection of pairs $(B,\psi)$ as
above is specified). We refer the reader to \cite[4.2]{BO} for the notion of ``A-set of centrally 
extended points'' which is a formalization of the cohomological data above in the special case 
when $\tilde \omega$ is trivial. Such notions are important since 
it is known that any module category over $\Vec_A^\omega$ is
equivalent as a module category to $\Coh(Y)$ where $Y$ is such cohomologically twisted
$A-$set, see \cite[Example 2.1]{Odr}. 
\end{example}

Let $\cM$ be a module category over an indecomposable multi-fusion category $\cC$. Then one
defines the dual category $\cC_\cM^*$ to be the category of all endofunctors of $\cM$ which commute
with the action of $\cC$, see \cite[4.2]{Omc}. The category $\cC_\cM^*$ has a natural monoidal
structure where the tensor product is given by the composition of functors.
It is known that the category $\cC_\cM^*$ is again 
an indecomposable multi-fusion 
category, see \cite[Theorem 2.18]{ENO} (this result fails if $k$ is allowed to have positive characteristic).

\begin{example} \label{dualex}
Let $\cC=\Vec_A$ and let $\cM$ be as in Example \ref{modex} (i). In this case the category
$\cC_\cM^*$ is precisely the category $\Coh_A(Y\times Y)$ from Example \ref{multex} (iii).
Thus using module categories from Example \ref{modex} (ii) we get a cohomologically
twisted version of the category $\Coh_A(Y\times Y)$. We will use similar notation
$(\Vec_A^\omega)^*_\cM=\Coh_{A,\omega}(Y\times Y)$ where it is understood that
the $A-$set $Y$ is cohomologically twisted as in Example \ref{modex} (ii).
We note that the that direct summands $\{ \be_i\}_{i\in I}$ in the decomposition of the unit object $\be \in
\Coh_A(Y\times Y)$ are precisely the projection functors from the module category $\cM$ 
to its indecomposable
direct summands; in particular the set $I$ is in natural bijection with the set of such summands. 
\end{example}

The category $\cC^*_\cM$ consists of endofunctors of $\cM$; thus it acts in an obvious way on $\cM$. 
The following result justifies the terminology:

\begin{theorem}[see Remark 2.19 in \cite{ENO}] \label{doublec}
Let $\cC$ be an indecomposable multi-fusion category and let $\cM$ be a module
category over $\cC$. Then $\cC_\cM^*$ is also indecomposable multi-fusion and the natural
functor $\cC \to (\cC_\cM^*)^*_\cM$ is an equivalence of tensor categories.
\end{theorem}

Let $F: \cC \to \cD$ be a tensor functor between indecomposable multi-fusion categories.
We say that $F$ is {\em injective} if it is fully faithful and {\em surjective} if it is dominant,
that is any object of $\cD$ is contained in $F(X)$ for suitable $X\in \cC$. Now let
$\cM$ be a module category over $\cD$. Then $\cM$ can be considered as a module category
over $\cC$ and we have a natural dual tensor functor $F^*:\cD^*_\cM \to \cC^*_\cM$. It is shown
in \cite[5.7]{ENO} that this duality interchanges injective and surjective functors. 

\begin{example} \label{surjF}
Let $A\xrightarrow{f}\bar A$ be a surjective homomorphism of finite groups
and let $\tilde \omega$ be a 3-cocycle representing class $\omega \in H^3(\bar A,k^\times)$ such
that $f^*(\omega)$ is zero element of $H^3(A,k^\times)$. Then any 2-cochain $\psi$ such
that $\partial \psi =f^*(\tilde \omega)$ defines a tensor structure on the functor $F: \Vec_A\to
\Vec_{\bar A}^\omega$ sending $L_g$ to $L_{f(g)}$. The functor $F$ is surjective.
Conversely, it is easy to see that any surjective tensor functor $\Vec_A\to \cC$ where $\cC$ is
a fusion category is isomorphic to the one of this form.
\end{example} 

It is easy to see that the category $\cC^*_\cM$ is fusion if and only if the module category $\cM$ is
not a nontrivial direct sum of module categories over $\cC$, that is $\cM$ is {\em indecomposable} 
over $\cC$.
We have the following consequence of the discussion above:

\begin{corollary}[see Lemma 3.1 in \cite{LO}] \label{cco}
Let $\cM$ be a module category over $\Vec_A$ and let $\cC$ be a full multi-fusion subcategory 
of $(\Vec_A)^*_\cM$ such that $\cM$ is indecomposable over $\cC$.
Then there exists a surjective functor $F: \Vec_A\to \Vec_{\bar A}^\omega$
such that the action of $\Vec_A$ on $\cM$ factors through $F$ and such that
$\cC=(\Vec_{\bar A}^\omega)^*_\cM \subset (\Vec_A)^*_\cM$.
\end{corollary}

\begin{proof} Let $G: \cC \to (\Vec_A)^*_\cM$ be the embedding functor; clearly it is injective.
 Then the dual functor $G^*: ((\Vec_A)^*_\cM)^*_\cM \to \cC^*_\cM$ is surjective. By Theorem
 \ref{doublec} we have $((\Vec_A)^*_\cM)^*_\cM =\Vec_A$ and the category $\cC^*_\cM$ is fusion.
 By Example \ref{surjF} the result follows.
\end{proof}

The module categories over a fixed indecomposable multi-fusion category $\cC$ form a 2-category,
where the morphisms are the module functors and 2-morphisms are the natural transformations
of the module functors. This 2-category is semisimple in the following sense: for any module
categories $\cM_1$ and $\cM_2$ the category of module functors $\Fun_\cC(\cM_1,\cM_2)$ 
from $\cM_1$ to $\cM_2$ is finite semisimple, see \cite[Theorem 2.18]{ENO}). It is clear that
the composition of functors makes $\Fun_\cC(\cM_1,\cM_2)$ into a module category over
$\Fun_\cC(\cM_1,\cM_1)=\cC^*_{\cM_1}$. One shows that the 2-functor $\Fun_\cC(\cM,?)$ is a 2-equivalence
of 2-categories of module categories over $\cC$ and over $\cC^*_\cM$, see \cite[Proposition 2.3]{Odr}
or \cite{Mu1}. For example this means that there is one to one correspondence between the
module categories over $\Coh_{A,\omega}(Y\times Y)$ and over $\Vec_A^\omega$; moreover
to compute the module functors between the module categories over $\Coh_{A,\omega}(Y\times Y)$
we can compute the module functors between the corresponding module categories over 
$\Vec_A^\omega$.

\begin{example} Let $\cM_1=\cM(B_1,\psi_1)$ and $\cM_2=\cM(B_2,\psi_2)$ be the module categories over
$\cC=\Vec_A$ as in Example \ref{modex} (ii). Assume that $\psi_1$ and $\psi_2$ are both trivial. 
Then the category $\Fun_\cC(\cM_1,\cM_2)$ identifies with the category $\Coh_{B_1}(A/B_2)$ of
$B_1-$equivariant vector bundles on the $B_1-$set $A/B_2$, see e.g. \cite[Proposition 3.2]{Odr}.
\end{example}

In general it is difficult to find a number of simple objects in the category $\Fun_\cC(\cM_1,\cM_2)$.
Here is a special case when this is possible to do. Let $\be =\oplus_{i\in I}\be_i$
be the decomposition of the unit object of $\cC$ into simple summands. Let $\cC\ot \be_i$ be the
full subcategory of $\cC$ consisting of objects $X$ such that $X\ot \be_i\simeq X$. It is clear
that $\cC \ot \be_i$ is stable under the left multiplications by objects from $\cC$. Thus $\cC \ot \be_i$
is a module category over $\cC$. Note that for any module category $\cM$ over $\cC$ the Grothendieck
group $K(\cM)$ is naturally a module over the Grothendieck ring $K(\cC)$.

\begin{lemma}[Lemma 3.4 in \cite{LO}] \label{214}
Let $\cM$ be a module category over a multi-fusion
category $\cC$. Then the number of simple objects in the category $\Fun_\cC(\cC \ot \be_i,\cM)$ equals
the dimension of $\Hom_{K(\cC)}(K(\cC\ot \be_i),K(\cM))$.
\end{lemma}

\subsection{Drinfeld center} \label{drc}
One of the most important constructions in the theory of tensor categories
is the construction of {\em Drinfeld center}, see \cite{JS, Ma}. 
One definition in the spirit of Section \ref{mdc} is as follows. 
A monoidal category $\cC$ acts on itself by left and right multiplications,
so $\cC$ is a {\em bimodule category} over itself. Then the Drinfeld center $\cZ(\cC)$ of $\cC$ is 
the category of endofunctors of $\cC$ commuting with these actions. The composition makes 
$\cZ(\cC)$ into a monoidal category, but we have more structure here: $\cZ(\cC)$ is naturally a 
braided tensor category, see \cite{JS,Ma}. It is easy to see (\cite[2.3]{Odr})
that our definition is equivalent to the classical one: the objects of $\cZ(\cC)$ are pairs
$(X,\phi)$ where $X$ is an object of $\cC$ and $\phi$ is an isomorphism of functors
$X\ot ?\simeq ?\ot X$ satisfying some natural conditions, see \cite{JS,Ma,Mu2}. We have a natural
{\em forgetful functor} $\cZ(\cC)\to \cC$ sending $(X,\phi)$ to $X$. The right adjoint of this functor
(if it exists) is called the {\em induction functor}. 

It is known that the Drinfeld center of an indecomposable multi-fusion category is a fusion category,
see \cite[Theorem 2.15]{ENO} or \cite{Mu2}; in particular the induction functor exists in this case.  
Another important
property is the Morita invariance: for a module category $\cM$ we have a natural tensor
equivalence $\cZ(\cC^*_\cM)\simeq \cZ(\cC)$, see \cite[Corollary 2.6]{Odr} or \cite{Mu2}. 

\begin{example} \label{drex}
Recall that $\Coh_{A,\omega}(Y\times Y)$ is $(\Vec_A^\omega)^*_\cM$ for suitable $\cM$.
Thus we get a somewhat surprising result: $\cZ(\Coh_{A,\omega}(Y\times Y))$ does not depend on $Y$
and is equivalent to $\cZ(\Vec_A^\omega)$.
\end{example}

\section{Harish-Chandra bimodules} \label{HCb}
\subsection{Basic definitions}
Let $\fg$ be a complex semisimple Lie algebra. Let $U(\fg)$ be the universal enveloping algebra of $\fg$ and let $Z(\fg)\subset U(\fg)$ be the center of $U(\fg)$. Recall that a {\em central character} is a
homomorphism $\chi: Z(\fg)\to \BC$.
For a central character $\chi$
we have two sided ideal $U(\fg)\Ker(\chi)\subset U(\fg)$ and we will set 
$U_\chi:=U(\fg)/U(\fg)Ker(\chi)$.

Recall that for a $U(\fg)-$bimodule $M$ one defines
an adjoint $\fg-$action by the formula $ad(x)m:=xm-mx$; 
we will denote by $M_{ad}$ the space $M$ with this
action of $\fg$.

\begin{definition} A $U(\fg)-$bimodule $M$ is called $ad(\fg)$-algebraic if $M_{ad}$ can be
decomposed into a direct sum of finite dimensional $\fg-$modules. 
We say that an $ad(\fg)$-algebraic $U(\fg)-$bimodule $M$ is {\em Harish-Chandra bimodule}
if it is finitely generated as $U(\fg)-$bimodule.
\end{definition} 

\begin{remark} The definitions of Harish-Chandra bimodules in the literature (see e.g. \cite{Ja,Hu,So}) 
differ slightly from each other with $ad(\fg)-$algebraicity being the crucial part.
\end{remark}

\begin{example} \label{nounit}
 Consider $U(\fg)$ as $U(\fg)-$bimodule. Then the Poincar\'e-Birkhoff-Witt 
(or PBW) filtration
on $U(\fg)$ is stable under the adjoint action. Hence $U(\fg)$ is a Harish-Chandra bimodule.
Since $U(\fg)\ot U(\fg)$ is Noetherian any subquotient of Harish-Chandra bimodule is again
Harish-Chandra bimodule. Hence $I$ and $U(\fg)/I$ are Harish-Chandra bimodules 
for any two sided ideal $I\subset U(\fg)$. In particular $U_\chi$ is a Harish-Chandra bimodule.
\end{example} 

\begin{remark} \label{gK}
Assume that $\fg$ is a complexification of the Lie algebra of a real semisimple Lie
group $G_\BR$ with a maximal compact subgroup $K$. It was shown by Harish-Chandra that
the study of continuous representations of $G_\BR$ to a large extent reduces to the study of 
so called $(\fg,K)-$modules (or Harish-Chandra modules). We recall that a $(\fg,K)-$module
is a finitely generated $U(\fg)-$module endowed with compatible locally finite
action of $K$, see e.g. \cite[2.1(a)]{Vo}; thus this is a purely algebraic object.
The notion of Harish-Chandra bimodule is a special case of this when we take 
in the place of  $G_\BR$ a complex simply connected Lie group with 
Lie algebra $\fg$ considered as a real Lie group (so the complexified Lie algebra is isomorphic
to $\fg \oplus \fg$). 
\end{remark}

The following well known result is of crucial importance for this note:

\begin{lemma} If $M$ and $N$ are Harish-Chandra bimodules, then so is $M\ot_{U(\fg)} N$.
\end{lemma}

\begin{proof} It is immediate from definitions that the canonical surjection $M\ot N\to M\ot_{U(\fg)} N$
commutes with the adjoint action. Hence $M\ot_{U(\fg)} N$ is $ad(\fg)-$algebraic. 

Let $M_0\subset M$ be a finite dimensional $ad(\fg)-$invariant subspace of $M$ generating
$M$ as $U(\fg)-$bimodule. It is easy to see that $M_0$ generates $M$ as left $U(\fg)-$module
and as right $U(\fg)-$module. Let $N_0\subset N$ be a similar subspace of $N$. Then 
the image of $M_0\ot N_0$ clearly generates $M\ot_{U(\fg)} N$.
\end{proof}

Let $\cH$ denote the category of Harish-Chandra bimodules (where the morphisms are
homomorphisms of bimodules). The tensor product over $U(\fg)$ with the obvious associativity
isomorphisms makes $\cH$ a tensor category with the unit object $U(\fg)$, see Example \ref{nounit}.
However this category has some unpleasant properties: the endomorphism algebra of $U(\fg)$ 
identifies with $Z(\fg)$, so the $\Hom-$spaces are infinite dimensional in general. 

\begin{remark} It is easy to see that for any $K$ as in Remark \ref{gK} the tensor product 
$M\ot_{U(\fg)} N$ of a Harish-Chandra bimodule
$M$ and $(\fg,K)-$module $N$ is again $(\fg,K)-$module. In other words, the category of Harish-Chandra bimodules acts naturally on the category of $(\fg,K)-$modules.
\end{remark}

\subsection{Irreducible Harish-Chandra bimodules} \label{BeGe}
For a central character $\chi$ let $\cH^\chi$ be the full subcategory of $\cH$ consisting
of bimodules $M$ such that $M\Ker(\chi)=0$ (in other words, the right action of $Z(\fg)$ on $M$ 
factorizes through $\chi$).
A very precise description of category $\cH^\chi$
was given by Bernstein and Gelfand in \cite{BG}.  
This description is based on the category $\cO$ of $\fg-$modules 
introduced by Bernstein, Gelfand and Gelfand
in \cite{BGG}. We refer the reader to \cite{Hu} for the basic definitions and results on the category $\cO$.

Recall that for any {\em weight} $\lambda$ one defines the {\em Verma module} $M(\lambda)\in \cO$.
The center $Z(\fg)$ acts on $M(\lambda)$ via central character $\chi_\lambda$. 
It follows from Harish-Chandra's theorem (see e.g. \cite[1.10]{Hu}) that any central
character arises in this way; moreover $\chi_\lambda =\chi_\mu$
if and only if there exists an element $w$ of the {\em Weyl group} $W$ such that 
$w(\lambda +\rho)-\rho=\mu$ where $\rho$ is the sum of the fundamental weights.
Thus for any central character $\chi$ there exists a {\em dominant} (see \cite[3.5]{Hu}) 
weight $\lambda$ such that $\chi=\chi_\lambda$.
From now on we will restrict ourselves to the case of {\em regular integral}
central characters $\chi$ (this means that $\chi=\chi_\lambda$ where $\lambda$ is a highest weight
of a finite dimensional representation of $\fg$). For example
$Z(\fg)$ acts on the trivial $\fg-$module via the regular integral character $\chi_0$.

\begin{theorem}[Theorem 5.9 in \cite{BG}] \label{BiG}
 Assume that the weight $\lambda$ is regular,
integral, and dominant.
The functor $M\mapsto M\ot_{U(\fg)}M(\lambda)$
is an equivalence of the category $\cH^\chi$ and the subcategory of $\cO$ consisting of modules
with integral weights.
\end{theorem}

As a consequence we see that any object of the category $\cH^\chi$ has finite length (this holds
with no restrictions on $\chi$, see e.g. \cite[Satz 6.30]{Ja}). Also the simple objects in the category
$\cH^\chi$ are labeled by the integral weights (see \cite[Proposition 5.4]{BG} for the general case).

Now we consider the left action of $Z(\fg)$. 
For two central characters $\chi_1$ and $\chi_2$ let ${}^{\chi_1}\cH^{\chi_2}$ be the category 
of Harish-Chandra bimodules $M$ such that $\Ker(\chi_1)M=M\Ker(\chi_2)=0$.
We also set $\cH(\chi):={}^{\chi}\cH^{\chi}$. It is clear that for $M\in {}^{\chi_1}\cH^{\chi_2}$ and
$N\in {}^{\chi_3}\cH^{\chi_4}$ we have $M\ot_{U(\fg)}N\in {}^{\chi_1}\cH^{\chi_4}$ and 
$M\ot_{U(\fg)}N=0$ unless $\chi_2=\chi_3$. In particular, the category 
$\cH(\chi)$ is a tensor category with unit object $U_\chi$.

For a regular integral $\chi_2$ Theorem \ref{BiG} implies that the category ${}^{\chi_1}\cH^{\chi_2}$
is nonzero if and only if $\chi_1$ is integral. Moreover one shows using Theorem \ref{BiG}
that the categories $\cH(\chi)$ are tensor equivalent for various regular integral $\chi$.
Furthermore, we have

\begin{corollary} Let $\chi=\chi_\lambda$ where $\lambda$ is regular,
integral, and dominant. The simple objects of $\cH(\chi)$ are naturally labeled
by the elements of $W$: for any $w\in W$ there exists $M_w\in \cH(\chi)$ such that
$M_w\ot_{U(\fg)}M(\lambda)$ is the irreducible $\fg-$module with highest weight
$w(\lambda +\rho)-\rho$.
\end{corollary}

\subsection{Associated varieties} \label{assv}
In this section we identify $\fg^*$ and $\fg$ via the Killing form. Let $G$ be the complex connected 
adjoint algebraic group with the Lie algebra $\fg$. An element $x\in \fg$ is {\em nilpotent}
if $ad(x): \fg \to \fg$ is nilpotent. Let $\cN \subset \fg$ be the {\em nilpotent cone}, that is the 
set of all nilpotent elements. Clearly $\cN$ is a closed $G-$invariant subvariety of $\fg$.
It is a classical fact (see \cite{K}) that $\cN$ is a finite union of $G-$orbits, $\cN=\sqcup \BO$.
The $G-$orbits appearing in $\cN$ are called {\em nilpotent orbits}. For a nilpotent orbit
$\BO$ let $\bar \BO \subset \cN$ be its Zariski closure; clearly $\bar \BO$ is a union of nilpotent orbits.

The associated varieties (see e.g. \cite{Vo}) provide 
a convenient measure of
``size'' of a Harish-Chandra bimodule. Let $M$ be a Harish-Chandra bimodule.
Then there exists a finite dimensional $ad(\fg)-$invariant subspace $M_0\subset M$ generating $M$
as a left $U(g)-$module. Then the PBW filtration on $U(\fg)$ induces a compatible filtration on $M$
(note that this filtration is $ad(\fg)-$invariant, so it is compatible with both left and right $U(\fg)-$actions).
The associated graded $gr M$ with respect to this filtration is a left module over $gr U(\fg)=S(\fg)$.
Let $V(M)$ be the support of this module in $\fg\simeq \fg^*=Spec(S(\fg))$. The following properties
of $V(M)$ are easy to verify (see e.g. \cite{Vo}):

(1) $V(M)$ is independent of the choice of $M_0$;

(2) $V(M)$ is invariant under the adjoint action of $G$;

(3) for an exact sequence $0\to M_1\to N\to M_2\to 0$ we have $V(N)=V(M_1)\cup V(M_2)$;

(4) $V(M_1\ot_{U(\fg)}M_2)\subset V(M_1)\cap V(M_2)$;

(5) for $M\in \cH^\chi$ we have $V(M)\subset \cN$.

The following result of Joseph is fundamental:

\begin{theorem}[\cite{Jo}, see also \cite{BoB,Vo}] \label{Joirr}
Assume that $M\in \cH$ is irreducible. Then $V(M)=\bar \BO$ for some 
nilpotent orbit $\BO$.
\end{theorem}

Assume that $M\in \cH(\chi)$ where $\chi$ is regular integral. It follows from the results of 
\cite{BV1,BV2,Jo}
that $V(M)=\bar \BO$ where $\BO$ is {\em special} nilpotent orbit in the sense of Lusztig, (all the
nilpotent orbits are special in type $A$; however this is not the case in other types); and
all special nilpotent orbits can be obtained in this way.

The theory of associated varieties is closely related with the theory of Kazhdan-Lusztig cells,
see \cite{KL}. Namely let us introduce the following equivalence relation on the Weyl group $W$:
$u\sim v$ if $V(M_u)=V(M_v)$. Then $W$ is partitioned into equivalence classes labeled by
the special nilpotent orbits. It follows from the results of \cite{BV1,BV2,Jo} that this partition coincides 
with the partition of $W$ into {\em two sided cells} as defined in \cite{KL}. In particular 
the set of two sided cells is in natural bijection with the set of special nilpotent orbits.

\subsection{Cell categories} \label{cellc}
Let $\chi$ be a regular integral central character and let $\BO$ be
a nilpotent orbit. We define full subcategories $\cH(\chi)_{\le \bar \BO}$ and $\cH(\chi)_{< \bar \BO}$
as follows: $M\in \cH(\chi)_{\le \bar \BO}$ (respectively, $M\in \cH(\chi)_{< \bar \BO}$) if and only if 
$V(M)\subset \bar \BO$ (respectively, $V(M)\subset \bar \BO -\BO$). It follows easily from
the properties of associated varieties that $\cH(\chi)_{\le \bar \BO}$ and $\cH(\chi)_{< \bar \BO}$ 
are Serre subcategories of $\cH(\chi)$;
also $\cH(\chi)_{\le \bar \BO}$ is closed under the tensor product
and the tensor product of bimodules from $\cH(\chi)_{< \bar \BO}$ and $\cH(\chi)_{\le \bar \BO}$ 
is contained in $\cH(\chi)_{< \bar \BO}$.

We define $\tilde \cH(\chi)_\BO$ to be the Serre quotient category 
$\cH(\chi)_{\le \bar \BO}/\cH(\chi)_{< \bar \BO}$ (note that the category $\tilde \cH(\chi)_\BO$ 
is nonzero if and only if the nilpotent orbit $\BO$ is special).
Let $\cH(\chi)_\BO\subset \tilde \cH(\chi)_\BO$ be the full
subcategory consisting of the semisimple objects in $\tilde \cH(\chi)_\BO$. 
The tensor product $\ot_{U(\fg)}$ descends
to a well defined tensor product functor $\ot$ on the category $\tilde \cH(\chi)_\BO$ endowed 
with the associativity constraint.

\begin{theorem}[see \cite{BFO2, Lo2, LO}] \label{hc multi}
The restriction of $\ot$ to $\cH(\chi)_\BO\subset \tilde \cH(\chi)_\BO$ takes values in
 the subcategory $\cH(\chi)_\BO$. 
 Moreover, the category $\cH(\chi)_\BO$ is an indecomposable multi-fusion category.
\end{theorem}

We will explain some ideas of the proof of Theorem \ref{hc multi} in Section \ref{W}. 
In the same time we will give a precise description of the cell category $\cH(\chi)_\BO$ as a 
multi-fusion category. For now we will explain that the category $\cH(\chi)_\BO$ does contain the
unit object. We recall (see e.g. \cite[1.9]{Ja}) that a two sided ideal $I\subset U(\fg)$ is {\em primitive} 
if it is the annihilator of an irreducible $\fg-$module. It follows from Schur's lema  that for a primitive 
ideal $I$ the intersection $I\cap Z(\fg)=\Ker(\chi)$ for some central character $\chi$; let
$Pr_\chi$ be the set of all such primitive ideals (this set is finite; it is explicitly known in all cases
thanks to the deep work of Joseph \cite{Jo1}). 

Let $I\in Pr_\chi$. It was proved by Joseph \cite{Jo} that $V(U(\fg)/I)=\bar \BO$ for some special nilpotent
orbit $\BO$ (this result is closely related with Theorem \ref{Joirr}, see \cite[Corollary 4.7]{Vo}).
 It is also known that for any ideal
$I'\supset I, I'\ne I$ the dimension of $V(U(\fg))/I')$ is strictly smaller than the dimension of $V(U(\fg))/I)$,
see \cite[3.6]{BK}. Therefore $U(\fg)/I$ contains a unique simple sub-bimodule $M_I$; moreover $V(U(\fg)/I)=V(M_I)=\bar \BO$ and $V((U(\fg)/I)/M_I)\subset \bar \BO -\BO$. In other words
$U(\fg)/I\simeq M_I$ in the category $\tilde \cH(\chi)_\BO$; in particular $U(\fg)/I\in \cH(\chi)_\BO$.
Also for two distinct $I,J\in Pr_\chi$ with $V(U(\fg)/I)=V(U(\fg)/J)=\bar \BO$ we have
$U(\fg)/I\ot_{U(\fg)}U(\fg)/J=U(\fg)/(I+J)$ whence $V(U(\fg)/I\ot_{U(\fg)}U(\fg)/J)\subset \bar \BO -\BO$.
Equivalently $U(\fg)/I\ot_{U(\fg)}U(\fg)/J=M_I\ot M_J=0$ in the category $\cH(\chi)_\BO$.

It is well known that for a simple Harish-Chandra
bimodule $M\in {}^{\chi_1}\cH^{\chi_2}$ there exist $I\in Pr_{\chi_1}$ and $J\in Pr_{\chi_2}$ such
that $I$ is the annihilator of $M$ considered as a left $U(\fg)-$module and $J$ is the annihilator
of $M$ considered as a right $U(\fg)-$module; moreover $V(M)=V(U(\fg)/I)=V(U(\fg)/J)$,
see e.g. \cite[7.7, 17.8]{Ja}. Clearly $U(\fg)/I\ot_{U(\fg)}M=M\ot_{U(\fg)}U(\fg)/J=M$. 
Let $Pr_\chi(\bar \BO)\subset Pr_\chi$ consists of $I$ with $V(U(\fg)/I)=\bar \BO$.
It follows from the above
that $$\be=\bigoplus_{I\in \Pr_\chi(\bar \BO)}U(\fg)/I=\bigoplus_{I\in \Pr_\chi(\bar \BO)}M_I$$ is the unit object of $\cH(\chi)_\BO$. Again there is an important connection with the theory of Kazhdan-Lusztig cells
\cite{KL}: it follows from the results of \cite{BV1, BV2, Lub} that two elements $u,v\in W$ are
in the same {\em left cell} if and only if there exists $\be_i$ such that $M_u\ot \be_i\simeq M_u$ and
$M_v\ot \be_i\simeq M_v$. In particular the set $\Pr_\chi$ is in bijection with the set of left cells in $W$.

\section{Actions of cell categories} \label{W}
\subsection{Whittaker modules} \label{Walg}
Let $e\in \fg$ be a nilpotent element. By the Jacobson-Morozov theorem
we can pick $h,f\in \fg$ such that $e,f,h$ is an $sl_2-$triple, that is $[h,e]=2e, [h,f]=-2f, [e,f]=h$. Then
$\fg$ decomposes into eigenspaces for $ad(h)$:
$$\fg =\bigoplus_{i\in \BZ}\fg(i),\;\; \fg(i)=\{ x\in \fg | [h,x]=ix\}.$$
In particular $e\in \fg(2)$ and $f\in \fg(-2)$. Using the Killing form $({},{})$ on $\fg$ one defines a skew-symmetric bilinear form $x,y\mapsto (e,[x,y])$ on the space $\fg(-1)$; it turns out that this form is
non-degenerate. Pick a lagrangian subspace $l\subset \fg(-1)$ and set 
$\m=\m_l=l\oplus \bigoplus_{i\le -2}\fg(i)$. Then $\xi(x)=(x,e)$ is a Lie algebra homomorphism 
$\m \to \BC$. Let $\m_\xi$ be the Lie subalgebra of $U(\fg)$ spanned by $x-\xi(x), x\in \m$.

\begin{definition}[\cite{Mo}]
We say that $\fg-$module is {\em Whittaker} if the action of $\m_\xi$  on it is
locally nilpotent.
\end{definition}

Let $\Wh$ be the category of Whittaker modules (this is a full Serre subcategory of category of
$\fg-$modules). We have a functor from $\Wh$ to vector spaces
$$M \mapsto \{ m\in M | xm=\xi(x)m,\; \forall x\in \m \}.$$
Let $U(\fg,e)$ be the algebra of endomorphisms of this functor; thus the functor above upgrades
to a functor $\text{Sk}: \Wh \to U(\fg,e)-mod$. An important result proved by Skryabin \cite{Sk}
(see also \cite{GanG} and \cite{Lo1}) is that this functor is an equivalence of categories.
Thus we call $\text{Sk}$ the Skryabin equivalence.

\begin{remark} The algebras $U(\fg,e)$ are {\em finite $W-$algebras} introduced by Premet
\cite{Pr}. We refer the reader to \cite{Loicm} for a nice survey of their properties. 
\end{remark}

A particularly important property of algebras $U(\fg,e)$ is that they do not depend on the choice
of lagrangian subspace $l$ (more precisely the algebras defined using different choices of $l$
are canonically isomorphic), see \cite{GanG}. In particular, the centralizer $Q$ of $e,h,f$ in $G$
acts naturally on $U(\fg,e)$, see \cite[2.6]{Loicm}.

\subsection{Irreducible finite dimensional representations of finite $W-$algebras} \label{Dodd etc}
Let $M\in \cH$ and $N\in \Wh$. For $x\in \fg$ and $m\ot n \in M\ot_{U(\fg)}N$ we have
$x(m\ot n)=ad(x)m\ot n+m\ot xn$. The subalgebra $\m$ consists of nilpotent elements, so 
$ad(x)$ is locally nilpotent for any $x\in \m$. Hence $M\ot_{U(\fg)}N\in \Wh$, in other words
the tensor category $\cH$ acts on the category $\Wh$. Let ${}^\chi \Wh$ be the full subcategory
of $M\in \Wh$ such that $Z(\fg)$ acts on $M$ through the central character $\chi$. Clearly
the action above restrict to an action of $\cH(\chi)$ on ${}^\chi \Wh$. 

We will be interested in the set $Y=Y(\chi)$ of isomorphism classes of irreducible modules $M$ in 
${}^\chi \Wh$ such that $\Sk(M)$ is finite dimensional. Since $\Sk$ is an equivalence, $Y$ is
also the set of irreducible finite dimensional representations of $W-$algebra $U(\fg,e)$. 
For any $M\in Y$ its annihilator is a primitive ideal of $U(\fg)$; thus we get a map
$\Ann_\chi : Y\to Pr_\chi$. It was proved by Premet \cite{Pr2} that any ideal $I$ in the image of this map
is contained in $Pr_\chi(\bar \BO)$ where $\BO=Ge$ is the nilpotent orbit containing $e$. Moreover, it was conjectured by Premet and proved by Losev \cite{Lo1} (see also \cite{Pr2,Pr3,Gi}) that 
any $I\in Pr_\chi(\bar \BO)$ is in the image of this map. Recall that the group
$Q$ acts on the algebra $U(\fg,e)$. Thus we get an action of $Q$ on the set $Y$. One shows
that the unit component $Q^0\subset Q$ acts trivially, so we get an action of the component
group $C(e)=Q/Q^0$ on $Y$ (it is well known that the group $C(e)$ is isomorphic to the component
group of the centralizer of $e$ in $G$ or, equivalently, $C(e)$ is equivariant fundamental group
of the orbit $\BO$). It was proved by Losev \cite{Lo2} that each fiber of the map $\Ann_\chi$
is exactly one $C(e)-$orbit in $Y$. We will seek for a precise description of these orbits.
Actually there is a little bit more information here. Let $M\in Y$ and let $Q_M\subset Q$ be its
stabilizer in the group $Q$. Then $Q_M$ acts projectively on $\Sk(M)$, so we have a cohomology
class in $H^2(Q_M,\BC^\times)$ describing this action. The data of the set $Y$ together with
$Q-$action and 2-cocycles above (which should be compatible in an obvious way) can be described
as the data of ``$Q-$set of centrally extended points'', see Example \ref{modex} (ii). 
Recall that in the special case when the group $Q$ is finite, precisely the same
data describe a structure of module category over $\Vec_Q$ on the category $\Coh(Y)$, see 
Example \ref{modex} (ii). In general, $\Coh(Y)$ acquires the structure of module category
over $\Vec_A$ for any finite subgroup $A\subset Q$.

Let ${}^\chi \Wh^f\subset {}^\chi \Wh$ be the full subcategory consisting of semisimple $N$
such that $\Sk(N)$ is finite dimensional. It was proved by Losev \cite{Lo2} that for 
$M\in \cH(\chi)_{\le \BO}$ and $N\in {}^\chi \Wh^f$ we have $M\ot_{U(\fg)}N\in
{}^\chi \Wh^f$. Moreover, $M\ot_{U(\fg)}N=0$ for $M\in \cH(\chi)_{< \BO}$. Thus the category
$\cH(\chi)_\BO$ acts on ${}^\chi \Wh^f$. On the other hand the group $Q$ acts on the category
${}^\chi \Wh^f$ (or rather on the equivalent category of $U(\fg,e)-$modules) via twisting: an element 
$g\in Q$ sends a $U(\fg,e)-$module to itself with the action of $U(\fg,e)$ twisted by an automorphism $g$.
One shows that these two actions commute. We pick a finite subgroup $A\subset Q$ which surjects
on $C(e)=Q/Q^0$ and restrict the above action of $Q$ to $A$. Then the category ${}^\chi \Wh^f$
is a module category over fusion category $\Vec_A$ (note that  ${}^\chi \Wh^f\simeq \Coh(Y)$,
and this is the same structure of the module category as in the previous paragraph). 
Since any $M\in \cH(\chi)_\BO$ produces a functor $M\ot_{U(\fg)}?: {}^\chi \Wh^f\to {}^\chi \Wh^f$
commuting with the action of $\Vec_A$ we get a canonical tensor functor
$$\cH(\chi)_\BO\to \Fun_{\Vec_A}({}^\chi \Wh^f,{}^\chi \Wh^f)=(\Vec_A)^*_{{}^\chi \Wh^f}=
\Coh_A(Y\times Y),$$ 
see Example \ref{dualex}.

\begin{theorem} \label{hcw} {\em (\cite{Lo2, LO})}
The functor $\cH(\chi)_\BO\to \Coh_A(Y\times Y)$ is fully faithful. 
\end{theorem}

\begin{remark} An important tool in the proof of Theorem \ref{hcw} is the notion of Harish-Chandra
bimodules for $W-$algebras introduced by Ginzburg \cite{Gi} and Losev \cite{Lo2}. 
It is possible to replace the group
$Q$ by the finite group $A$ since the action of $Q$ on $U(\fg,e)$ has the following property:
there is embedding of the Lie algebra $\fq$ of $Q$ to $U(\fg,e)$ (considered as a Lie algebra) such that
the differential of $Q-$action on $U(\fg,e)$ coincides with the adjoint action of $\fq$,
see \cite[1.1(1)]{Lo2}.
\end{remark} 

One consequence of Theorem \ref{hcw} is the fact that
the category $\cH(\chi)_\BO$ closed under the tensor product, see \cite[Corollary 1.3.2]{Lo2}.
This is a crucial step in the proof of Theorem \ref{hc multi}.
Moreover one shows that the module category ${}^\chi \Wh^f$ over $\cH(\chi)_\BO$ is indecomposable,
see \cite[Theorem 5.1]{LO}.
Thus we can apply Corollary \ref{cco} and get the following
\begin{corollary} \label{barA}
There is a quotient $\bar A$ of $A$ and $\omega \in H^3(\bar A,\BC^\times)$ such that the action
of $Vec_A$ on ${}^\chi \Wh^f$ factors through tensor functor $\Vec_A\to \Vec_{\bar A}^\omega$
and the action on ${}^\chi \Wh^f$ induces tensor equivalence $\cH(\chi)_\BO \simeq \Coh_{A,\omega}(Y\times Y)$.
\end{corollary}

It turns out that the quotient map $A\to \bar A$ always factorizes through $A\to Q\to Q/Q^0=C(e)$.
Thus $\bar A$ is naturally a quotient of the group $C(e)$. It was shown in \cite{LO} that $\bar A$
coincides with the {\em Lusztig's quotient} of $C(e)$ which was introduced by Lusztig 
\cite[Section 13]{Lub}. 
Also it was shown in \cite{LO} (see also \cite{BFO1}) that the class $\omega$ is trivial in almost
all cases. However it is not trivial in the case case of nilpotent orbits corresponding to so
called {\em exceptional} two sided cells, see \cite{Oex}.  

It follows from the results in Section \ref{hc=tr} below that 
the rational Grothendieck ring $K(\cH(\chi)_\BO)\ot \BQ$ is naturally
a quotient of the group algebra $\BQ[W]$. Thus for any module category $\cM$ over
$\cH(\chi)_\BO$ the rational Grothendieck group $K(\cM)\ot \BQ$ is naturally $W-$module. 
In the special case $\cM=\cH(\chi)_\BO \ot \be_i$ we obtain the {\em constructible representations} 
of $W$, see \cite[5.29]{Lub}; these representations are explicitly known. 
On the other hand let $\Spr$ be the {\em Springer
representation} of $W\times C(e)$ (this is top rational cohomology of the {\em Springer fiber} 
associated with $e\in \BO$ with the natural action of $C(e)$ and the action of $W$ defined by Springer
\cite{Spr}). It was proved by Dodd \cite{Dodd} that there is $W\times C(e)-$equivariant 
embedding $K(\Coh(Y))\ot \BQ \subset \Spr$. Using this result together with some results
by Lusztig
on Springer representation \cite{LuSpr} and Lemma \ref{214} the module category 
$\Coh(Y)$ over $\Vec_{\bar A}^\omega$ was explicitly determined in all cases in \cite{LO}. 
We recall that the indecomposable summands of $\Coh(Y)$ are naturally labeled by the simple summands $\be_i$ of $\be \in \cH(\chi)_\BO$, see Example \ref{dualex}.  Moreover, each such summand
is of the form $\cM(B_i,\psi_i)$, see Example \ref{modex} (ii). It was shown in \cite{LO} that we have
$C(e)-$equivariant isomorphism $\BQ [\bar A/B_i]\simeq \Hom_W(K(\cH(\chi)_\BO \ot \be_i)\ot \BQ,
\Spr(\BO))$; moreover this determines the subgroups $B_i\subset \bar A$ uniquely up to
conjugacy. It turned out that the subgroups $B_i$ precisely coincide with {\em Lusztig's subgroups}
\cite[Proposition 3.8]{Lulead}
associated to various left cells in $W$ (recall that the summands $\be_i$ are labeled by 
the left cells contained in the two sided cell corresponding to $\BO$, see Section \ref{cellc}). Note that the map $\Ann_\chi :Y\to \Pr_\chi(\bar \BO)$
has the following interpretation: for any $M\in Y$ there is a unique $\be_i$ such that
$\be_i\ot M\simeq M$ and $\Ann_\chi(M)$ is precisely 
the primitive ideal $I$ such that $\be_i=M_I$, see Section \ref{cellc}. This implies 
that the fiber of the map $\Ann_\chi$ over $I\in \Pr_\chi(\BO)$ such that $\be_i=M_I$ 
is precisely $C(e)-$set $\bar A/B_i$.

Now let us assume that the two sided cell corresponding to the orbit $\BO$ is not exceptional
(so the class $\omega$ is trivial). It follows from the computations described above that there
is one class $\psi \in H^2(\bar A,\BC^\times)$ such that the classes $\psi_i$ are just inverse 
images of $\psi$ under the embeddings $B_i\subset \bar A$, see \cite[Theorem 7.4]{LO}.
Equivalently, the class describing the projective action 
of $Q_M$ on $M\in Y$ is the inverse image of $\psi$ under the map $Q_M\subset Q\to C(e)\to \bar A$.
The recent results of Losev imply that the class $\psi$ is always trivial.
To prove this we can assume that $\fg$ is simple.
The result certainly holds if $H^2(\bar A,\BC^\times)=0$. It follows from the classification of
the nilpotent orbits that if $H^2(\bar A,\BC^\times)\ne 0$
then either $\fg$ is classical or $\fg$ is exceptional and $\bar A$ is the symmetric group on four
or five letters. In both cases there exists a 1-dimensional $U(\fg,e)-$module fixed by the action of $Q$:
for the classical $\fg$ this is  \cite[Theorem 1.2]{Lo1d} and for the exceptional $\fg$ one can use the
{\em generalized Miura transform} (see \cite[2.2]{Loicm}) since $e$ must be even in this case.
Thus we obtain the desired triviality of $\psi$ since
a projective 1-dimensional representation is equivalent to an actual representation.

\begin{remark} (i) There is a conjectural extension of the picture above to the case of
non-integral central characters $\chi$, see \cite[7.6]{LO} The computations suggest that 
in this case non-trivial
2-cocycles will arise quite often.

(ii) The results above give a description of the set $Y$ (we note that for the Lie algebras of type $A$
such a description is due to Brundan and Kleshchev \cite{BrK}). An immediate next question is what
are dimensions of the spaces $\Sk(M), M\in Y$, or, equivalently, what are dimensions of the
irreducible representations of $W-$algebras.  A complete answer to this question is given
in a recent paper \cite{Lo1d}; remarkably in the same time some old questions about the
{\em Goldie ranks} of the primitive ideals are resolved in {\em loc. cit.}
\end{remark}

\subsection{Harish-Chandra modules} \label{hcm}
 It would be interesting to investigate whether the ideas above 
apply to the categories $\cH^K$ of $(\fg,K)-$modules as in Remark \ref{gK}. Recall that we have a
{\em Cartan decomposition} $\fg =\fk \oplus \fp$ where $\fk$ is the complexified Lie algebra of $K$.
For a finitely generated $(\fg,K)-$module $M$ one defines its associated variety $V(M)\subset \fp$,
see \cite{Vo}. Let ${}^\chi \cH^K$ be the full subcategory of $\cH^K$ consisting of modules $M$
such that $Z(\fg)$ acts on $M$ via central character $\chi$. Then for $M\in {}^\chi \cH^K$
we have $V(M)\subset \fp \cap \cN$, see \cite[Corollary 5.13]{Vo}. 

Clearly the category $\cH(\chi)$ acts on ${}^\chi \cH^K$ via $\ot_{U(\fg)}$. 
Let us fix a nilpotent orbit $\BO$ and consider the Serre subcategories ${}^\chi \cH^K_{\le \BO}$ and 
${}^\chi \cH^K_{< \BO}$ consisting of 
$M\in {}^\chi \cH^K$ with $V(M)\subset \fp \cap \bar \BO$ and $V(M)\subset \fp \cap (\bar \BO -\BO)$. 
We can form the quotient category
${}^\chi \tilde \cH^K_{\BO}={}^\chi \cH^K_{\le \BO}/{}^\chi \cH^K_{< \BO}$; then $\ot_{U(\fg)}$ gives
us a bifunctor $\ot: \cH(\chi)_\BO \times {}^\chi \tilde \cH^K_{\BO}\to {}^\chi \tilde \cH^K_{\BO}$.
Let ${}^\chi \cH^K_{\BO}\subset {}^\chi \tilde \cH^K_{\BO}$ be the full subcategory of semisimple objects.

\begin{conjecture} \label{gKconj}
For $M\in \cH(\chi)_\BO$ and $N\in {}^\chi \cH^K_{\BO}$ we have
$M\ot N\in {}^\chi \cH^K_{\BO}$.
\end{conjecture}

Conjecture \ref{gKconj} would imply that ${}^\chi \cH^K_{\BO}$ is a module category over
$\cH(\chi)_\BO$. By Corollary \ref{barA} we have a tensor equivalence 
$\cH(\chi)_\BO=\Coh_{A,\omega}(Y\times Y)$ and by 
the results of Section \ref{mdc} we have a classification of all indecomposable module categories
over $\Coh_{A,\omega}(Y\times Y)$.
It would be very interesting to decompose the category ${}^\chi \cH^K_{\BO}$ and to
identify the indecomposable summands in terms of this classification.

\section{Sheaves} \label{CSh}
Let $F$ be an algebraically closed field of arbitrary characteristic. In this Section
we will consider consider
various classes of sheaves on algebraic varieties over $F$: $D-$modules ($F$ is of characteristic zero),
constructible sheaves in the classical topology ($F=\BC$), and constructible $l-$adic sheaves ($l$ is invertible in $F$). The corresponding categories of sheaves are $k-$linear where $k=F$, $k$ is
arbitrary of characteristic zero, and $k=\bar \BQ_l$ respectively.
Recall that the theories of such sheaves are parallel up to some extent. Thus we will not specify 
the kind of sheaves we deal with below unless this is necessary; the results are parallel in all
three setups.

\subsection{Convolution and Hecke algebra} Let $G$ be a semisimple algebraic group over $F$ of the same Dynkin type as $\fg$. Let $\Fl$ be the {\em flag variety} of $G$. We recall that $\Fl$ is a projective variety
which is a homogeneous space for $G$; furthermore the {\em Bruhat decomposition} gives a 
canonical bijection between $G-$orbits on $\Fl \times \Fl$ and the Weyl group $W$.
Let $D^b_G(\Fl \times \Fl)$ be the suitable
$G-$equivariant derived category of sheaves on $\Fl \times \Fl$, see e.g. \cite[2.2]{BFO2}.  The category
$D^b_G(\Fl \times \Fl)$ contains a natural abelian subcategory $\cP$ consisting of $D-$modules 
or perverse sheaves. The simple objects in the category $\cP$ are the intersection cohomology
complexes of closures of $G-$orbits on $\Fl \times \Fl$. This gives a natural bijection $w\to I_w$ 
between $W$ and the isomorphism classes of simple objects in $\cP$.

The category
$D^b_G(\Fl \times \Fl)$ has a natural monoidal structure with respect to tensor product given
by {\em convolution}, see e.g. \cite[2.4]{BFO2} (this construction is parallel to Example \ref{multex} (iii)).
 It follows from the Decomposition Theorem \cite{BBD}
that the convolution $I_{u}*I_{v}$ is isomorphic to a direct sum of shifted $I_w$:
$$I_u*I_v\simeq \bigoplus_{w\in W, i\in \BZ} n_{u,v}^w(i)I_w[i]$$
where the multiplicities $n_{u,v}^w(i)\in \BZ_{\ge 0}$.
Let $K(\cP)$ be the algebra over $\BZ[t, t^{-1}]$ which encodes the multiplicities $n_{u,v}^w(i)$ above:
the algebra has a basis $c_w$ and $c_u\cdot c_v=\sum_{w\in W, i\in \BZ} n_{u,v}^w(i)c_wt^i$. 
It is a classical result that the algebra $K(\cP)$ together with its basis $\{ c_w\}$ identifies with 
the {\em Hecke algebra} together with the {\em Kazhdan-Lusztig basis}, see e.g. \cite[2.5]{SpQ}. 
In particular, the multiplicities $n_{u,v}^w(i)$ are computable in principle.

Lusztig defined (see \cite{Lulead}) the {\em asymptotic Hecke algebra} $J$ in the following way:
for $w\in W$ let $a(w)=\max \{ i\in \BZ | n_{u,v}^w(i)\ne 0\, \mbox{for some}\, u,v\in W\}$. 
Let $J$ be a free $\BZ-$module with basis $t_w, w\in W$ endowed with multiplication
$t_ut_v=\sum_{w\in W}n_{u,v}^w(a(w))t_w$. It was shown by Lusztig that this multiplication
is associative and has a unit. Moreover there is a canonical isomorphism of associative algebras
$\BQ[W]\simeq J\ot \BQ$, see \cite[3.2]{Lulead}. 
Furthermore, for any subset $T\subset W$ let $J_T$ be the abelian
subgroup of $J$ spanned by $t_u, u\in T$. It is easy to see that there is a finest partition
$W=\sqcup C$ such that the decomposition $J=\bigoplus_CJ_C$ is a direct sum of algebras.
This partition is known to coincide with partition of $W$ into two sided Kazhdan-Lusztig cells,
see \cite[3.1]{Lulead}. 
It is also known that the function $a$ takes a constant value on any two sided cell $C$;
we will denote this value by $a(C)$.
 
The constructions above was categorified in \cite{Lutc}. Namely, for any two sided cell $C$ let
$\cJ_C\subset \cP$ be the full subcategory consisting of direct sums $I_w, w\in C$. 
The category $\cJ_C$ has a monoidal structure given by the {\em truncated convolution} $\bullet$,
see \cite{Lutc}. For example
$$I_u\bullet I_v\simeq \bigoplus_{w\in C} n_{u,v}^w(a(C))I_w.$$
Hence the assignment $I_w\mapsto t_w$ induces isomorphism of based rings
$K(\cJ_C)\simeq J_C$. It was shown in \cite{BFO1} (see also \cite{Lutrun}) that
the category $\cJ_C$ is rigid. As a consequence$\cJ_C$ is an indecomposable multi-fusion category. 

\subsection{$D-$modules and Harish-Chandra bimodules} \label{hc=tr}
In this section we assume that $F$ is of characteristic zero. The Beilinson-Bernstein theorem
\cite{BB} is a fundamental result in representation theory of $\fg$. It states that the category of
$U(\fg)-$modules with the trivial central character $\chi=\chi_0$ is equivalent to the category
of $D-$modules on $\Fl$. As a consequence one deduces that  the category of Harish-Chandra
bimodules $\cH(\chi_0)$ is equivalent to the category of $D-$modules $\cP$. However 
the Beilinson-Bernstein equivalence is not a tensor equivalence. Luckily it was shown 
in \cite{BeGi} that a composition of the Beilinson-Bernstein equivalence and a {\em long intertwining
functor} has a natural structure of tensor functor. Using a suitable truncation of this functor
the following result was shown in \cite[Corollary 4.5(b)]{BFO2}:

\begin{theorem} \label{joth}
 Let $C$ be the two-sided cell corresponding to a special nilpotent orbit $\BO$
(see Section \ref{assv}). Then there is a natural tensor equivalence $\cH(\chi_0)_\BO \simeq \cJ_C$.
\end{theorem}

Theorem \ref{joth} gives a useful information on the Grothendieck ring $K(\cH(\chi)_\BO)$
(we recall that for a regular integral character $\chi$ the category $\cH(\chi)_\BO$ is tensor
equivalent to $\cH(\chi_0)_\BO$, see Section \ref{BeGe}).
In particular we get a homomorphism $\BQ[W]\simeq J\ot \BQ \to J_C\ot \BQ=K(\cJ_C)\ot \BQ =
K(\cH(\chi_0)_\BO)\ot \BQ$ alluded to in Section \ref{Dodd etc}. In particular, 
we can consider $K(\cJ_C\ot \be_i)\ot \BQ$ as $W-$modules; these $W-$ representations 
are precisely the constructible representations discussed in {\em loc. cit.}

Conversely Theorem \ref{joth} combined
with Corollary \ref{barA} gives an explicit description of $\cJ_C$ as $\Coh_{\bar A}(Y\times Y)$ for the 
$D-$module version of the category $\cJ_C$. This implies similar description of $\cJ_C$ for other
categories of sheaves under the assumption that the ground field $F$ has characteristic 0. It was
shown in \cite{BFO1} (see also \cite{Oex} for the case of exceptional two sided cells) that the same description holds over a field $F$ of arbitrary characteristic.

\begin{remark} Theorem \ref{joth} was inspired by closely related results of Joseph \cite{Jotc}.
Also a similar and related connection between Kazhdan-Lusztig cells and Harish-Chandra
bimodules is contained in the work of Mazorchuk and Stroppel \cite{MaSt}.
\end{remark}

\subsection{Drinfeld center and character sheaves}
Lusztig introduced a very important class of {\em character sheaves} on the group $G$, see \cite{LuC}.
We recall the definition in the special case of {\em unipotent character sheaves}. Let 
$$X=\{ (b,b',g)\in
\Fl \times \Fl \times G | gb=b'\}$$ 
we have two projections $f: X\to \Fl \times \Fl, f(b,b',g)=(b,b')$ and 
$\pi : X\to G, \pi(b,b',g)=g$. Note that group $G$ acts on itself by conjugations and on $X$ via 
$h\cdot (b,b',g)=(hb,hb',hgh^{-1})$ and both maps $f,\pi$ are $G-$equivariant. Thus we have a functor
$\Gamma: D^b_G(\Fl \times \Fl)\to D^b_G(G), \Gamma =\pi_!f^*$. It follows from the Decomposition
Theorem that $\Gamma(I_w), w\in W$ is isomorphic to a direct sum of shifted simple 
$G-$equivariant sheaves on $G$; a simple $G-$equivariant sheaf is called a unipotent character
sheaf if it appears in such decomposition (possibly with some shift). Let $\cU$ be the set of
isomorphism classes of unipotent character sheaves on $G$; clearly this is a finite set.

One observes that the functor $\Gamma$ above has formal properties similar to the induction
functor from the monoidal category $D^b_G(\Fl \times \Fl)$ to its Drinfeld center, see Section \ref{drc}. Moreover,
it is possible to identify a suitable version of the Drinfeld center of $G-$equivariant sheaves on
$\Fl \times \Fl$ with suitable category of character sheaves. This was done in \cite{BFO2} using
the abelian tensor category of Harish-Chandra bimodules and in \cite{BN} using suitable
infinity categories. Furthermore applying a suitably truncated version of the same idea to the 
categories $\cJ_C$,
the following result was proved in \cite{BFO2} (for the field $F$ of characteristic zero)
and in \cite{Lutrun} (for the field $F$ of arbitrary characteristic):

\begin{theorem} \label{charac}
There is a partition $\cU=\sqcup_C \cU_C$ such that 
the Drinfeld center of the category $\cJ_C$ is naturally equivalent to 
the category of sheaves on $G$ which are direct sums of objects from $\cU_C$.
In particular we have a bijection 
$$\cU_C \leftrightarrow \{ \mbox{simple objects of the Drinfeld center of} \, \cJ_C\} .$$
\end{theorem}

The sets $\cU_C$ were defined by Lusztig in \cite[Section 16]{LuC3}. Recall that the category $\cJ_C$ is
tensor equivalent to $\Coh_{\bar A,\omega}(Y\times Y)$. Hence the Drinfeld center of $\cJ_C$
is equivalent to the Drinfeld center of $\Vec_{\bar A}^\omega$, see Example \ref{drex}. 
The resulting bijection
between $\cU_C$ and simple objects of $\cZ(\Vec_{\bar A}^\omega)$ 
conjecturally coincides with Lusztig's one from \cite[17.8.3]{LuC4} which gives us
a new approach to Lusztig's classification of character sheaves. 
On the other hand in \cite{Oex} character sheaves were used in order to determine
the associativity constraint in the categories $\cJ_C$ for exceptional cells $C$.

\subsection{Some generalizations} Many constructions described in this paper extend
to the case when the Weyl group $W$ is replaced by an arbitrary Coxeter group. An important
special case of the {\em affine Weyl groups} was considered in \cite{B,BO} following conjectures
made by Lusztig. In this case the counterparts of the cell categories are in one to one 
correspondence with all nilpotent orbits of $\fg$ and are of the form
$\Coh_Q(Y\times Y)$ where the reductive group $Q$ is the same as in Section \ref{Walg}
(note that the resulting categories are typically not multi-fusion categories since they
have infinitely many simple objects). 
The set $Y$ has a natural interpretation in terms of {\em non-restricted representations} of
$\fg$ over fields of positive characteristic, see \cite{BM}.

Using recent deep results by Elias and Williamson \cite{EW} on Soergel bimodules \cite{Sb}
Lusztig defined in \cite[Section 10]{Lutrun}
the counterparts of the cell categories for an arbitrary Coxeter group $W$ (note that
these categories sometimes are not even tensor categories since they lack the unit object;
however this is not very serious). It would be
very interesting to identify the resulting categories. For example in the case of the dihedral group
of order 10 one finds a cell category which contains a fusion subcategory with two simple objects 
$\be, X$ and the tensor product $X\ot X=\be \oplus X$. This implies that the cell category is not
of the form $\Coh_{A,\omega}(Y\times Y)$ in this case.


\end{document}